\documentclass[preprint, 10pt, english]{elsarticle}
\usepackage{amsthm}
\usepackage{amsmath}
\usepackage{latexsym, amssymb}
\usepackage{txfonts}
\usepackage{mathtools}
\usepackage{color}
\usepackage{babel}
\usepackage[all]{xy}
\usepackage[capitalise]{cleveref}

\newtheorem{thm}{Theorem}[section] 

\newtheorem{cor}[thm]{Corollary}

\newtheorem{lem}[thm]{Lemma}

\theoremstyle{definition}
\newtheorem{rem}[thm]{Remark}

\newcommand\operA[2]{{\if!#2!\operatorname{#1}\else{\operatorname{#1}_{#2}^{\phantom{I}}}\fi}} 

\def\norm{{\operatorname{N}}}

\newcommand{\Trace}[1][]{\if!#1!\operatorname{Tr}\else{\operatorname{Tr}_{#1}^{\phantom{I}}}\fi} 

\long\def\forget#1\forgotten{{}} %

\def\({\left(}
\def\){\right)}

\newcommand\LAY[3][]{{\begin{array}{c}\mbox{#2} \if#1!{}\else{+}\fi \\ \mbox{#3}\end{array}}}

\makeatletter

\def\ps@pprintTitle{%
 \let\@oddhead\@empty
 \let\@evenhead\@empty
 \def\@oddfoot{}%
 \let\@evenfoot\@oddfoot}

\newcommand{\bigperp}{%
  \mathop{\mathpalette\bigp@rp\relax}%
  \displaylimits
}

\newcommand{\bigp@rp}[2]{%
  \vcenter{
    \m@th\hbox{\scalebox{\ifx#1\displaystyle2.1\else1.5\fi}{$#1\perp$}}
  }%
}
\makeatother

\renewcommand{\geq}{\geqslant}
\renewcommand{\leq}{\leqslant}

\newif\iffurther
\furtherfalse

\journal{??}

\begin{document}
\begin{frontmatter}

\title{Linkage and Essential $p$-Dimension}

\author{Adam Chapman}
\address{School of Computer Science, Academic College of Tel-Aviv-Yaffo, Rabenu Yeruham St., P.O.B 8401 Yaffo, 6818211, Israel}
\ead{adam1chapman@yahoo.com}

\begin{abstract}
We prove that two cyclically linked $p$-algebras of prime degree become inseparably linked under a prime to $p$ extension if and only if the essential $p$-dimension of the pair is 2.
We conclude that the essential $p$-dimension of pairs of cyclically linked $p$-algebras is 3 by constructing an example of a pair that does not become inseparably linked under any prime to $p$ extension.
\end{abstract}

\begin{keyword}
Linkage, Cyclic Algebras, Fields of positive characteristic, Essential Dimension, Kato-Milne Cohomology
\MSC[2020] 16K20 (primary); 13A35, 19D45 (secondary)
\end{keyword}
\end{frontmatter}

\section{Introduction}

Given a prime integer $p$ and a field $F$ of characteristic $p$, a cyclic algebra of degree $p$ over $F$ is an algebra of the form
$[\alpha,\beta)_{p,F}=F\langle i,j : i^p-i=\alpha, j^p=\beta, jij^{-1}=i+1\rangle$
for some $\alpha \in F$ and $\beta \in F^\times$.
We say that two given such algebras are cyclically linked if they share a cyclic field extension of degree $p$ of the center (and thus can be written as symbols sharing the left slot), and inseparably linked if they share a purely inseparable field extension of the center of degree $p$ (and thus can be written as symbols sharing the right slot).

It is known that inseparable linkage implies cyclic linkage and that the opposite is in general not true (see \cite{Chapman:2015}).
The question of what additional condition ensures the converse statement was answered in the special case of $p=2$ (see \cite{ChapmanGilatVishne:2017}), in which case the two algebras $[\alpha,\beta)_{2,F}$ and $[\alpha,\gamma)_{2,F}$ are inseparably linked if and only if the class $\alpha \operatorname{dlog} \beta \wedge \operatorname{dlog}(\gamma)$ in $H_2^3(F)$ is trivial.
In \cite{Chapman:2020b}, it was shown that when $p=3$, if the class of $\alpha \operatorname{dlog} \beta \wedge \operatorname{dlog}(\gamma)$ in $H_3^3(F)$ is trivial, then $[\alpha,\beta)_{3,K}$ and $[\alpha,\gamma)_{3,K}$ are inseparably linked for some extension $K/F$ of $[K:F] \leq 2$. 
In this paper, we prove that for any prime $p$, if the class of $\alpha \operatorname{dlog} \beta \wedge \operatorname{dlog}(\gamma)$ in $H_p^3(F)$ is trivial, then $[\alpha,\beta)_{p,K}$ and $[\alpha,\gamma)_{p,K}$ are inseparably linked for some extension $K/F$ of $p \nmid [K:F]$, and use this to show that two cyclically linked algebras over a field containing an algebraically closed field of characteristic $p$ is of essential $p$-dimension 2 if and only if the two become inseparably linked under restriction to some prime to $p$ extension. We conclude that the essential dimension of pairs of cyclically linked algebras of degree $p$ over field extensions of an algebraically closed field is 3.

\section{Preliminaries}

\subsection{Essential Dimension}
Setting a field $k$ (often assumed to be algebraically closed), and given a covariant functor $\mathcal{F}$ from the category of field extensions $F$ of $k$ to the category of sets, the essential dimension of $A \in \mathcal{F}(F)$, denoted $\operatorname{ED}_k(A)$, is the minimal transcendence degree of a field $E/k$ with $E \subseteq F$ to which $A$ descends, i.e., for which there exists $B \in \mathcal{F}(E)$ such that $A=B \otimes F$ (see \cite{BerhuyFavi:2003}).
The essential $p$-dimension of $A$, denoted $\operatorname{ED}_k(A;p)$, is the minimal essential dimension of $A \otimes L$ where $L$ ranges over the prime to $p$ extensions of $F$. The essential dimension (respectively, essential $p$-dimension) of $\mathcal{F}$, denoted $\operatorname{ED}_k(\mathcal{F})$ (respectively, $\operatorname{ED}_k(\mathcal{F};p)$), is the supremum on the essential dimension (respectively, essential $p$-dimension) of $A$ where $A$ ranges over all elements in $\mathcal{F}(F)$ and $F$ ranges over all field extensions of $k$.
Our interest in this paper is the functor $\operatorname{CLA}$ mapping $F$ to the set of isomorphism classes of pairs $(A,B)$ of cyclically linked algebras of degree $p$ over $F$. The isomorphisms are of central simple $F$-algebras, which means they act as the identity map on the center $F$.

\subsection{Kato-Milne Cohomology}
Given a field $F$ of $\operatorname{char}(F)=p$,
the space of $n$-fold differential forms $c d \beta_1 \wedge \dots \wedge d \beta_n$ is denoted by $\Omega^n F$.
Note that when $\beta \neq 0$, $d\beta/\beta=d\log \beta$.
Consider the map $\wp : \Omega^n F \rightarrow \Omega^n F/d \Omega^{n-1} F$ taking $\alpha d\log \beta_1 \wedge \dots \wedge d\log \beta_n$ to $(\alpha^p-\alpha) d\log \beta_1 \wedge \dots \wedge d\log \beta_n$.
The cokernel of this map is denoted $H_p^{n+1}(F)$ (see \cite{Kato:1982} and \cite[Chapter 9]{GilleSzamuely:2006}).
This group is connected to many important groups, from the filtration groups of the Witt group of fields of characteristic 2 when $p=2$, to the $p$-torsion of the Brauer group: ${_pBr}(F) \cong H_p^2(F)$ given by $[\alpha,\beta)_{p,F} \mapsto \alpha d\log \beta$. This paper is part of the attempt to see how the classes in the higher cohomology groups tell us something about the behaviour of central simple algebras.

\section{Inseparable Linkage and $H_p^3(F)$}

We start by citing two useful theorems:
\begin{thm}[{\cite[Th\'eor\`eme 6]{Gille:2000}}]\label{Gille}
The class of $\alpha d\log(\beta)\wedge d\log(\gamma)$ is trivial in $H_p^3(F)$ if and only if $\gamma$ is the norm of an element in the algebra $[\alpha,\beta)_{p,F}$.
\end{thm}

\begin{thm}[{\cite[Chapter 4, Theorem 1.13]{Pfister:1995}}]\label{Pfister}
Given system of $n$ homogeneous polynomial forms of degrees prime to $p$ in $n-1$ variables over a field $F$, there exists a nontrivial solution to the system in some prime to $p$ extension of $F$.
\end{thm}

We are now ready to see how the trivial class in $H_p^3(F)$ affects inseparable linkage.

\begin{thm}
Given a field $F$ of $\operatorname{char}(F)=p$, if $\alpha d\log \beta \wedge d\log \gamma$ is trivial in $H_p^3(F)$, then $[\alpha,\beta)_{p,K}$ and $[\alpha,\gamma)_{p,K}$ are inseparably linked for some $K/F$ of $p\nmid [K:F]$.
\end{thm}

\begin{proof}
By Theorem \ref{Gille}, $\gamma$ is the norm of some element $t$ in $[\alpha,\beta)_{p,F}$.
Set $s_1,\dots,s_{p-1},s_p$ to be the characteristic functions of $A=[\alpha,\beta)_{p,F}$, i.e., the homogeneous polynomial functions of degrees $1,\dots,p-1,p$ from $A$ to $F$ such that $y^p+s_1(y)y^{p-1}+\dots+s_{p-1}(y) y+s_p(y)=0$ for any $y\in A$. In particular, $-s_1$ is the trace form and $-s_p$ is the norm.
Recall that the first generator $i$ of $A$ satisfies $i^p-i=\alpha$. Write $f=x_0+x_1 i+\dots+x_{p-1} i^{p-1}$.
By Theorem \ref{Pfister}, the system $s_1(ft)=s_2(ft)=\dots=s_{p-1}(ft)=0$ of $p-1$ equations of degrees $1,2,\dots,p-1$ in $p$ variables $x_0,x_1,\dots,x_{p-1}$ has a nontrivial solution in some $K/F$ of $p\nmid [K:F]$.
The element $z=ft$ in $[\alpha,\beta)_{p,K}$ thus satisfies $z^p=\norm_{K[i]/K}(f)\gamma$, and so $[\alpha,\beta)_{p,K}=[\delta,\norm_{K[i]/K}(f)\gamma)_{p,K}$ for some $\delta \in K$.
However, $[\alpha,\gamma)_{p,K}=[\alpha,\norm_{F[i]/F}(f)\gamma)_{p,K}$ as well, and therefore, the algebras are inseparably linked.
\end{proof}

\begin{cor}
If $[\alpha,\beta)_{p,F}$ and $[\alpha,\gamma)_{p,F}$ share all cyclic maximal subfields, then $[\alpha,\beta)_{p,K}$ and $[\alpha,\gamma)_{p,K}$ are inseparably linked for some $K/F$ of $p\nmid [K:F]$.
\end{cor}

\begin{proof}
By \cite[Corollary 3.3]{ChapmanDolphin:2019}, the class of $\alpha d\log\beta \wedge d\log \gamma$ in $H_p^3(F)$ is trivial, because the algebras share all cyclic maximal subfields. Then apply the previous theorem.
\end{proof}

\begin{cor}
If $F$ is a $p$-special field (i.e., has no prime to $p$ field extensions) of $\operatorname{char}(F)=p$ with $H_p^3(F)=0$, then cyclic linkage coincides with inseparable linkage.
\end{cor}

\begin{rem}
One may wonder if $([\alpha,\beta)_{p,F},[\alpha,\gamma)_{p,F}) \mapsto \alpha d\log \beta d\log \gamma \in H_p^3(F)$ is a well-defined cohomological invariant. Unfortunately, it is not. Since $[\alpha,\beta)_{p,F}=[-\alpha,\beta^{-1})_{p,F}$ and $[\alpha,\gamma)_{p,F}=[-\alpha,\gamma^{-1})_{p,F}$, the associated class in $H_p^3(F)$ is also $-\alpha d\log \beta^{-1} \wedge d\log \gamma^{-1}$, which is $-\alpha d\log \beta d\log \gamma$, which is not equal to $\alpha d\log \beta \wedge d\log \gamma$ when $p\geq 3$. For $p=2$ it is a well-defined cohomological invariant, as described in \cite{ChapmanGilatVishne:2017}.\end{rem}
\section{Linkage and Essential Dimension}

Set $k$ to be an algebraically closed field of $\operatorname{char}(k)=p$.
Note that any field $F$ of transcendence degree $m$ over $k$ is a $C_m$ field (see \cite{Lang:1952}), which means that every homogeneous polynomial equation of degree $d$ in more than $d^m$ variables over $F$ has a nontrivial solution.
By \cite{ArasonBaeza:2010}, $H_p^{m+1}(F)=0$.

\begin{thm}
Given a field extension $F$ of $k$ and cyclically linked algebras $[\alpha,\beta)_{p,F}$ and $[\alpha,\gamma)_{p,F}$, the essential $p$-dimension of the pair is $\leq 2$ if and only if they become inseparably linked under the restriction to some prime to $p$ extension of $F$.
\end{thm}

\begin{proof}
If the essential $p$-dimension of $([\alpha,\beta)_{p,F},[\alpha,\gamma)_{p,F})$ is $\leq 2$, there exists a prime to $p$ extension $K/F$ such that $[\alpha,\beta)_{p,K}$ and $[\alpha,\gamma)_{p,K}$ descend to $[a,b)_{p,E}$ and $[a,c)_{p,E}$ for some $k \subseteq E \subseteq K$ where $E/k$ is of transcendence degree at most 2.
Then, $a d\log b \wedge d\log c$ is trivial in $H_p^3(E)$, because the latter is a trivial group, and thus $[a,b)_{p,E}$ and $[a,c)_{p,E}$ are inseparably linked, and therefore, their restrictions to $K$ are inseparably linked.

In the opposite direction, if $[\alpha,\beta)_{p,K}$ and $[\alpha,\gamma)_{p,K}$ are inseparably linked for some prime to $p$ extension $K/F$, then there exist $a,b \in K$ such that $[\alpha,\beta)_{p,K}=[a,b)_{p,K}$ and $[\alpha,\gamma)_{p,K}=[a,b-1)_{p,K}$ by \cite[Theorem 4.7]{ChapmanFlorenceMcKinnie:2023}.
The pair thus descends to the field $E=k(a,b)$, which is of transcendence degree at most 2.
\end{proof}

\begin{rem}Note that in the theorem above, the moment one of the algebras is a division algebra, the essential $p$-dimension must be at least 2, because fields of transcendence degree 1 have trivial Brauer groups. In fact, essential $p$-dimension 1 for such pairs is impossible, which leaves three options for each pair: essential $p$-dimension 3 when the algebras do not become inseparably linked under any prime to $p$ extension, essential $p$-dimension 2 when they do but at least one of the algebras is non-split, and 0 when the two algebras are split.
\end{rem}

\begin{lem}
Given a field $F$ of $\operatorname{char}(F)=p$ and $\alpha,\beta \in F$ such that $[\alpha,\beta)_{p,F}$ is a division algebra, the algebras $[\alpha,\beta)_{p,F(x)}$ and $[\alpha,x)_{p,F(x)}$ do not become inseparably linked under any restriction to a prime to $p$ extension of $F(x)$.
\end{lem}

\begin{proof}
Consider the $x$-adic valuation on $K=F(x)$.
It is easy to see that the algebra $[\alpha,x)_{p,FK}$ is division, because the value of $x$ is 1 whereas the norms from $K[\wp^{-1}(\alpha)]/K$ are of values divisible by $p$.
This valuation extends naturally to this algebra, by the restriction to $F(\!(x)\!)$ over which it becomes henselian.
The algebra $[\alpha,\beta)_{p,K}$ is unramified.
Both algebras are defectless, and thus every common subfield of theirs is defectless.

Let $L/K$ be a prime to $p$ extension.
By \cite[Proposition 4.21]{TignolWadsworth:2015}, there exists an extension of the valuation to $L$ such that $p \nmid [\Gamma_L:\Gamma_{K}]$.
Therefore, by Morandi's theorem (see \cite[Theorem 3.43]{TignolWadsworth:2015}), $[\alpha,x)_{L,p}$ is a division algebra to which the valuation extends and its value group is $\Gamma_L+\frac{1}{p}\mathbb{Z}=\frac{1}{p}\Gamma_L$.
The valuation also extends to the algebra $[\alpha,\beta)_{p,L}$ with value group $\Gamma_L$. Both algebras are still defectless with respect to the valuation on $L$.
Since the intersection of the value groups is $\Gamma_L$, the algebras only share maximal subfields whose value groups are $\Gamma_L$, and because each such subfield is defectless, its residue algebra is a degree $p$ extension of the residue $\overline{L}$ of $L$. In particular, if this common subfield is a purely inseparable extension of $L$, its residue algebra is still a degree $p$ purely inseparable extension of $\overline{L}$. The residue algebra of $[\alpha,x)_{p,L}$ is $\overline{L}[\wp^{-1}(\alpha)]$, and so every common subfield has $\overline{L}[\wp^{-1}(\alpha)]$ as a residue algebra, and in particular, cannot be a purely inseparable extension of $L$. Therefore, $[\alpha,x)_{p,L}$ and $[\alpha,\beta)_{p,L}$ are not inseparably linked.
\end{proof}

\begin{cor}
The essential dimension and essential $p$-dimension of $\operatorname{CLA}$ with respect to $k$ are equal to 3.
\end{cor}

\begin{proof}
They cannot be greater than 3 because the pair $([\alpha,\beta)_{p,F},[\alpha,\gamma)_{p,F})$ always descends to $k(\alpha,\beta,\gamma)$.
It remains to prove that there exists a pair of algebras that remain not inseparably linked under any prime to $p$ extension, and thus their essential $p$-dimension is not $\leq 2$, which leaves the option of 3 alone.
By the previous lemma, $[\alpha,\beta)_{p,K}$ and $[\alpha,\gamma)_{p,K}$ remain division algebras under restriction to any prime to $p$ extension $L/K$ where $K=F(\gamma)$ and $F=k(\alpha,\beta)$. Therefore, $\operatorname{ED}_k(([\alpha,\beta)_{p,K},[\alpha,\gamma)_{p,K});p)=3$. Hence, $\operatorname{ED}_k(\operatorname{CLA};p)=3$.
Since the essential dimension here is at most three and no less than the essential $p$-dimension, we have $\operatorname{ED}_k(\operatorname{CLA})=3$ as well.
\end{proof}

For $p=2$, the last corollary provides a characteristic 2 analogue for \cite[Theorem 1.3 (c)]{CerneleReichstein:2015}.
In the latter, the authors were studying the essential dimension of triples of quaternion $F$-algebras $Q_1,Q_2,Q_3$ where $Q_1 \otimes Q_2 \otimes Q_3$ is split for $\operatorname{char}(F)\neq 2$. This is the same as the essential dimension of pairs of linked quaternion algebras, because $Q_1 \otimes Q_2 \otimes Q_3$ is split if and only if $Q_1$ and $Q_2$ are linked and $Q_3$ is the unique quaternion algebra Brauer equivalent to $Q_1 \otimes Q_2$.
\section*{Acknowledgements}
The author is grateful to the two referees for the helpful comments and suggestions.
\bibliographystyle{abbrv}
\bibliography{bibfile}
\end{document}